\documentclass[12pt]{amsart}
\usepackage{amsfonts, amsbsy, amsmath, amssymb}

\newtheorem{thm}{Theorem}[section]

\newtheorem{cor}[thm]{Corollary}

\newtheorem{fac}[thm]{Fact}

\newtheorem{thm-con}[thm]{Theorem-Conjecture}
\numberwithin{equation}{section}

\theoremstyle{definition}

\begin{document}

\title[A note on the permutation behaviour of the polynomial $g_{n,q}$]{A note on the permutation behaviour of the polynomial $g_{n,q}$}

\author[Neranga Fernando]{Neranga Fernando}
\address{Department of Mathematics,
Northeastern University, Boston, MA 02115}
\email{w.fernando@northeastern.edu}

\begin{abstract}
Let $q=4$ and $k$ a positive integer. In this short note, we present a class of permutation polynomials over $\Bbb F_{q^{3k}}$. We also present a generalization.  
\end{abstract}



\maketitle

\section{Introduction}

For a prime power $q$, let $\Bbb F_q$ denote the finite field with $q$ elements. Let $e$ be a positive integer. A polynomial $f \in \Bbb F_q[x]$ is called a {\it permutation polynomial} (PP) of $\Bbb F_q$ if it induces a permutation of $\Bbb F_q$. Define 
$$S_{k,q} = x+x^q+\cdots +x^{q^{k-1}} \in \Bbb F_p[x],$$
where $k$ is a positive integer and $p=\textnormal{Char}\Bbb F_q$. When $q$ is fixed, we write $S_{k,q}=S_k$. Note that $S_e=\textnormal{Tr}_{q^e/q}$, where $\textnormal{Tr}_{q^e/q}$ is the {\it trace} function from $\Bbb F_{q^e}$ to $\Bbb F_q$.

PPs over $\Bbb F_{q^{3k}}$, where $q$ is a power of $2$ have been studied in \cite{Fernando-Hou-Lappano-FFA-2013, Hou13, Yuan-Ding}. In \cite{Hou13}, Hou confirmed a class of PPs over $\Bbb F_{4^{3k}}$ that was conjectured in \cite{Fernando-Hou-Lappano-FFA-2013}. In \cite{Yuan-Ding}, Yuan and Ding generalized Hou's result and obtained several more classes of PPs of the form $L(x)+S_{2k}^a+S_{2k}^b$ over $\Bbb F_{q^{3k}}$, where $q$ is a power of $2$. In this paper, we present a new class of PPs over $\Bbb F_{4^{3k}}$, where $k>0$ is an even integer. 

Table 1 in \cite{Fernando-Hou-FFA-2015} contains PPs over $\Bbb F_{4^e}$ defined by the polynomial $g_{n,q}$ which was first introduced in \cite{Hou-JCTA-2011} by Hou. Permutation property of $g_{n,q}$ was studied in \cite{Hou-FFA-2012, Fernando-Hou-Lappano-FFA-2013}. We call a triple of integers $(n,e;q)$ {\em desirable} if $g_{n,q}$ is a PP of $\Bbb F_{q^e}$. We refer the reader to \cite{Hou-FFA-2012} for more background of the polynomial $g_{n,q}$. 

Desirable triples obtained from computer searches were presented in \cite{Fernando-Hou-FFA-2015, Fernando-Hou-Lappano-FFA-2013, Hou-FFA-2012}. In \cite{Fernando-Hou-FFA-2015}, Hou and the author of the present paper discovered several new classes of PPs which explained several entries of Table 3 in \cite{Fernando-Hou-Lappano-FFA-2013}. In this short note, we explain yet another entry of Table 1 in \cite{Fernando-Hou-FFA-2015}. 

The note is organized as follows. 

In Section 2, we present a new class of PPs over $\Bbb F_{4^{3k}}$. In section 3, we explain an entry of Table 1 in \cite{Fernando-Hou-FFA-2015} using the results obtained in Section 2. In Section 4, we present a generalization.


\section{A class of permutation polynomials over $\Bbb F_{q^{3k}}$}

We recall two facts: 

\begin{fac}\label{f1} (\cite[Theorem 7.7]{Lidl-Niederreiter-97}) Let $f \in \Bbb F_q[x]$. Then $f$ is a PP of $\Bbb F_q$ if and only if for all $a\in \Bbb F_q^*$, $\sum_{x\in\Bbb F_q}\zeta_p^{\textnormal{Tr}_{q/p}(af(x))}=0$, where $p=\textnormal{Char}\Bbb F_q$ and $\zeta_p=e^{2\pi i/p}$.
\end{fac}

\begin{fac}\label{f2} (\cite[Lemma 6.13]{Fernando-Hou-Lappano-FFA-2013}) Let $p$ be a prime and $f:\Bbb F_p^n\to \Bbb F_p$ a function. If there exists a $y\in \Bbb F_p^n$ such that $f(x+y)-f(x)$ is a nonzero constant for all $x\in\Bbb F_p^n$, then $\sum_{x\in\Bbb F_p^n}\zeta_p^{f(x)}=0$.
\end{fac}

The following theorem is the main theorem of this paper. 

\begin{thm}\label{T1}
Let $q=4$, $e=3k$ where $k>0$ is an even integer.
Then 
$$g=S_{k+1}^2+S_{2k}^{q^k+1}$$
is a PP of $\Bbb F_{q^e}$.
\end{thm}

\begin{proof}

By Fact~\ref{f1}, it suffices to show that
\[
\sum_{x\in\Bbb F_{q^e}}(-1)^{\text{Tr}_{q^e/2}(ag(x))}=0
\]
for all $0\ne a\in\Bbb F_{q^e}$. We write $\text{Tr}=\text{Tr}_{q^e/2}$.

\medskip

{\bf Case 1.} Assume $\text{Tr}_{q^e/q^k}(a)\ne0$. By Fact~\ref{f2}, it suffices to show that there exists $y\in \Bbb F_{q^e}$ such that $\text{Tr}\bigl(ag(x+y)-ag(x)\bigr)$ is a nonzero constant for all $x\in\Bbb F_{q^e}$.

Let $y\in\Bbb F_{q^k}^*$. Then for $x\in\Bbb F_{q^e}$, we have
\[
\begin{split}
&\text{Tr}\bigl(ag(x+y)-ag(x)\bigr)\cr
=\,&\text{Tr}\bigl[a\bigl(S_{k+1}^2(x+y)+S_{2k}^{q^k+1}(x+y)-S_{k+1}^2(x)-S_{2k}^{q^k+1}(x)\bigr)\bigr]\cr
=\,&\text{Tr}(a\,S_{k+1}^2(y))\kern 3.8cm (S_{2k}(y)=0\ \text{since}\ y\in\Bbb F_{q^k})\cr
=\,&\text{Tr}(a\,S_{k-1}^{2q}(y)) \cr
=\,&\text{Tr}(b\,S_{k-1}(y)) \kern 3.8cm (a=b^{2q}) \cr
=\,&\text{Tr}\bigl(b\,(y+y^q+\cdots+y^{q^{k-2}})\bigr)\cr
=\,&\text{Tr}\bigl(y\,(b^{q^{3k}}+b^{q^{3k-1}}+\cdots+b^{q^{2k+2}})\bigr)\cr
=\,&\text{Tr}_{q^k/2}\bigl(y\,(c^{q^{2k+2}}+\cdots+c^{q^{3k}})\bigr), \kern 1cm  \text{where}\,\, c=\text{Tr}_{q^e/q^k}(b) \neq 0.
\end{split}
\]

Since 
\[
\begin{split}
&\text{gcd}({\tt x}^{2k+2}+{\tt x}^{2k+3}+\cdots+{\tt x}^{3k},\; {\tt x}^{k}+1)\cr
=\,&\text{gcd}(1+{\tt x}+\cdots+{\tt x}^{k-2},\; {\tt x}^{k}+1)\cr
=\,&1,
\end{split}
\]

$c^{q^{2k+2}}+\cdots+c^{q^{3k}}\neq 0$, and hence there exists $y\in\Bbb F_{q^k}$ such that $\text{Tr}\bigl(ag(x+y)-ag(x)\bigr) = \text{Tr}\bigl(y\,(b^{q^{3k}}+b^{q^{3k-1}}+\cdots+b^{q^{2k+2}})\bigr)$ is a nonzero constant.

{\bf Case 2.}  Assume $\text{Tr}_{q^e/q^k}(a)=0$. Then $a=b+b^{q^k}$ for some $b\in\Bbb F_{q^e}$. Since $a\ne 0$, we have $b\notin\Bbb F_{q^k}$. For $x\in\Bbb F_{q^e}$, we have
\[
\begin{split}
\text{Tr}\bigl(ag(x)\bigr)\,&=\text{Tr}\bigl(bg(x)+b^{q^k}g(x)\bigr)\cr
&=\text{Tr}\bigl(bg(x)+bg(x)^{q^{2k}}\bigr)\cr
&=\text{Tr}\bigl(b(g(x)+g(x)^{q^{2k}})\bigr).
\end{split}
\]

Note that 
$$ S_{2k}^{q^k} + S_{2k}^{q^{2k}} \equiv S_{2k} \pmod{{\tt x}^{q^e}-{\tt x}}$$
and
$$S_{k+1}^2+S_{k+1}^{2q^{2k}} \equiv x^2+x^{2q^k}+S_{2k}^2+S_{2k}^{2q^k} \pmod{{\tt x}^{q^e}-{\tt x}} $$ 

Then we have
\begin{equation}\label{e1}
\begin{split}
g({\tt x})+g({\tt x})^{q^{2k}}\,&\equiv S_{k+1}^2+S_{2k}^{q^k+1}+S_{k+1}^{2q^{2k}}+S_{2k}^{q^{2k}+1}\pmod{{\tt x}^{q^e}-{\tt x}}\cr
&= S_{k+1}^2+S_{k+1}^{2q^{2k}}+S_{2k}\cdot(S_{2k}^{q^k}+S_{2k}^{q^{2k}})\cr
&\equiv x^2+x^{2q^k}+S_{2k}^2+S_{2k}^{2q^k}+S_{2k}^2 \pmod{{\tt x}^{q^e}-{\tt x}}\cr
&=({\tt x}+{\tt x}^{q^{k}}+S_{2k}^{q^k})^2\cr
&=(S_{3k}+S_{k}^q)^2\cr
&\equiv (S_{2k}^{q^{k+1}})^2\pmod{{\tt x}^{q^e}-{\tt x}}.   
\end{split}
\end{equation}

So for $x\in\Bbb F_{q^e}$,
\[
\begin{split}
\text{Tr}\bigl(ag(x)\bigr)\,&=\text{Tr}\bigl(bS_{2k}(x)^{2q^{k+1}}\bigr)\cr
&=\text{Tr}\bigl(cS_{2k}(x)\bigr)\kern 3.3cm (b=c^{2q^{k+1}})\cr
&=\text{Tr}\bigl(c(x+x^q+\cdots+x^{q^{2k-1}})\bigr)\cr
&=\text{Tr}\bigl(x(c^{q^{3k}}+c^{q^{3k-1}}+\cdots+c^{q^{k+1}})\bigr).
\end{split}
\]
Since
\[
\begin{split}
&\text{gcd}({\tt x}^{k+1}+{\tt x}^{k+2}+\cdots+{\tt x}^{3k},\; {\tt x}^{3k}+1)\cr
=\,&\text{gcd}(1+{\tt x}+\cdots+{\tt x}^{2k-1},\; {\tt x}^{3k}+1)\cr
=\,&{\tt x}^k+1,
\end{split}
\]
we see that for $z\in \Bbb F_{q^{3k}}$, 
\[
z^{q^{3k}}+z^{q^{3k-1}}+\cdots+z^{q^{k+1}}=0 \Leftrightarrow z\in\Bbb F_{q^k}.
\]
Since $c\notin \Bbb F_{q^k}$, we have $c^{q^{3k}}+c^{q^{3k-1}}+\cdots+c^{q^{k+1}}\ne 0$. Therefore 
\[
\sum_{x\in\Bbb F_{q^e}}(-1)^{\text{Tr}(ag(x))}=\sum_{x\in\Bbb F_{q^e}}(-1)^{\text{Tr}(x(c^{q^{3k}}+c^{q^{3k-1}}+\cdots+c^{q^{k+1}}))}=0.
\]
\end{proof}


\section{Polynomial $g_{n,q}$}

We first recall two facts: 

\begin{fac}\label{f5} (\cite[Eq.~4.1]{Hou-FFA-2012}.) 
For $m,q \geq 0$, we have 
$$g_{m+q^a,a} = g_{m+1,q} + S_a \cdot g_{m,q}.$$
\end{fac}

\begin{fac}\label{f4}  (\cite[Theorem~6.1]{Fernando-Hou-Lappano-FFA-2013}.)
Let $q\ge 4$ be even, and let
\[
n=1+q^{a_1}+q^{b_1}+\cdots+q^{a_{q/2}}+q^{b_{q/2}},
\]
where $a_i,b_i\ge 0$ are integers.
Then
$$g_{n,q}=\sum_iS_{a_i}S_{b_i}+\sum_{i<j}(S_{a_i}+S_{b_i})(S_{a_j}+S_{b_j}).$$
\end{fac}

\begin{cor}
Let $q=4$, $e=6$, and $n=65921=(q-3)q^0+2q^3+q^{4}+q^{8}$. Then 
\[
g_{n,q}\equiv S_3^2+S_{4}^{q^2+1}\pmod{{\tt x}^{q^e}-{\tt x}},
\]
and $g_{n,q}$ is a PP of $\Bbb F_{q^e}$.
\end{cor}

\begin{proof} We write $g_n$ for $g_{n,q}$.  Then by Fact~\ref{f4} we have
\[
\begin{split}
g_n\,&=S_3^2+S_4S_8\cr
&\equiv S_3^2+S_4 (S_6+S_2) \pmod{{\tt x}^{q^e}-{\tt x}}\cr
&= S_3^2+S_4 S_4^{q^2}\cr
&=S_3^2+S_{4}^{q^2+1}.
\end{split}
\]
Let $k=2$ in Theorem~\ref{T1}. Then It follows from Theorem~\ref{T1} that $g_n$ is a PP of $\Bbb F_{q^e}$.
\end{proof} 


\section{A Generalization}

The following theorem is a generalization of Theorem~\ref{T1}. 

\begin{thm}\label{T2}
Let $q$ be a power of $2$. Let $L\in \Bbb F_{q^{3k}}[x]$ be a $2$-linearized polynomial such that 
\begin{enumerate}
\item[(i)] $L$ permutes $\Bbb F_{q^{k}}$, and
\item[(ii)] $L+L^{q^{2k}} \equiv S_{2k}^2+S_{2k}^{2q^{k+1}} \pmod{{\tt x}^{q^{3k}}-{\tt x}}$.
\end{enumerate}
Then $L+S_{2k}^{q^k+1}$ is a PP of $\Bbb F_{q^{3k}}$.
\end{thm}

\begin{proof}
Let $g=L+S_{2k}^{q^k+1}$. In Case 1 in Theorem~\ref{T1}, we choose $y\in\Bbb F_{q^k}$ such that $\text{Tr}_{q^k/2}\bigl(L(y)\,\text{Tr}_{q^e/q^k}(a)\bigr)\neq 0$. In Case 2, \eqref{e1} still holds because of condition $2$. 
\end{proof}


\end{document}